\documentclass{article}
\usepackage[utf8]{inputenc}
\usepackage{amsmath,amsthm,amssymb,amsfonts}
\usepackage{verbatim}
\usepackage{array}

\usepackage{hyperref}

\usepackage{tikz}
\usepackage{tkz-berge, tkz-graph}
\tikzset{vertex/.style={circle,draw,fill,inner sep=0pt,minimum size=1mm}}

\usepackage{stmaryrd} 
\usepackage{hyperref}

\usepackage[colorinlistoftodos]{todonotes}

\usepackage{tikz}
\usepackage{tkz-berge, tkz-graph}

\tikzset{vertex/.style={circle,draw,fill,inner sep=0pt,minimum size=1mm}}

\theoremstyle{plain}
\newtheorem{thm}{Theorem}
\newtheorem{lem}[thm]{Lemma}
\newtheorem{prop}[thm]{Proposition}
\newtheorem{definition}[thm]{Definition}

\newtheorem{remark}[thm]{Remark}
\newtheorem{question}[thm]{Question}
\theoremstyle{definition}

\newtheorem{exl}[thm]{Example}

\numberwithin{thm}{section}

\newcommand{\adj}{\leftrightarrow}
\newcommand{\adjeq}{\leftrightarroweq}

\DeclareMathOperator{\id}{id}
\def\N{{\mathbb N}}
\def\R{{\mathbb R}}

\DeclareMathOperator{\Fix}{Fix}
\newcommand{\Z}{\mathbb{Z}}

\title{Convexity and Freezing Sets in Digital Topology}
\author{Laurence Boxer
\thanks{Department of Computer and Information
Sciences, Niagara University, NY 14109, USA; and
Department of Computer Science and Engineering,
State University of New York at Buffalo. Email: boxer@niagara.edu}
}

\date{}

\begin{document}

\maketitle
\begin{abstract}
    We continue the study of freezing sets in
    digital topology, introduced in~\cite{BxFpSets}.
    We show how to find a minimal freezing set for
    a ``thick" convex disk~$X$ in the digital plane $\Z^2$. We
    give examples showing the significance of the
    assumption that~$X$ is convex.
    
    
\end{abstract}
\section{Introduction}
We often use a digital image as a mathematical
model of an object or a set of objects ``pictured"
by the image. Methods inspired by classical
topology are used to determine whether a digital image
has properties analogous to the topological properties
of a ``real world" object represented by the image.
The literature now contains considerable success
in adapting to digital topology notions from
classical topology such as {\em connectedness,
continuous function, homotopy, fundamental group,
homology, automorphism group}, et al.

However, the fixed point properties of a digital image
are often very different from those of the Euclidean
object modeled by the image. Knowledge of the
fixed point set $\Fix(f)$ of a continuous
self-map on a nontrivial topological space~$X$
rarely tells us much about 
$f|_{X \setminus \Fix(f)}$. By contrast, it
was shown in~\cite{bs19a,BxFpSets} that
knowledge of the fixed point set $\Fix(f)$ of a
digitally continuous self-map on a nontrivial 
digital image~$(X,\kappa)$ may tell us a great
deal about $f|_{X \setminus \Fix(f)}$. Indeed, if
$A$ is a subset of $X$ that is a ``freezing set"
and $A \subset \Fix(f)$, then $f$ is constrained to
be the identity function $\id_X$.

Some results concerning freezing sets were
presented in~\cite{BxFpSets}. In this paper,
we continue the study of freezing sets. In 
particular, we show how to find minimal
freezing sets for ``thick" convex
disks in the digital plane, and we give examples
showing the importance of the assumption of
convexity in our theorems.

\section{Preliminaries}
We use $\Z$ to indicate the set of integers and
$\R$ for the set of real numbers.
For a finite set~$X$, we denote by $\#X$ the number of
distinct members of~$X$.

\subsection{Adjacencies}
Material in this section is quoted or paraphrased
from~\cite{BxFpSets}.

The $c_u$-adjacencies are commonly used 
in digital topology.
Let $x,y \in \Z^n$, $x \neq y$, where we consider these points as $n$-tuples of integers:
\[ x=(x_1,\ldots, x_n),~~~y=(y_1,\ldots,y_n).
\]
Let $u \in \Z$,
$1 \leq u \leq n$. We say $x$ and $y$ are 
{\em $c_u$-adjacent} if
\begin{itemize}
\item there are at most $u$ indices $i$ for which 
      $|x_i - y_i| = 1$, and
\item for all indices $j$ such that $|x_j - y_j| \neq 1$ we
      have $x_j=y_j$.
\end{itemize}
Often, a $c_u$-adjacency is denoted by the number of points
adjacent to a given point in $\Z^n$ using this adjacency.
E.g.,
\begin{itemize}
\item In $\Z^1$, $c_1$-adjacency is 2-adjacency.
\item In $\Z^2$, $c_1$-adjacency is 4-adjacency and
      $c_2$-adjacency is 8-adjacency.
\item In $\Z^3$, $c_1$-adjacency is 6-adjacency,
      $c_2$-adjacency is 18-adjacency, and $c_3$-adjacency
      is 26-adjacency.
\end{itemize}

For $\kappa$-adjacent $x,y$, we write $x \adj_{\kappa} y$ or $x \adj y$ when $\kappa$ is understood.
We write $x \adjeq_{\kappa} y$ or $x \adjeq y$ to mean that either $x \adj_{\kappa} y$ or $x = y$.

We say $\{x_n\}_{n=0}^k \subset (X,\kappa)$ is a {\em $\kappa$-path} (or a {\em path} if $\kappa$ is understood)
from $x_0$ to $x_k$ if $x_i \adjeq_{\kappa} x_{i+1}$ for $i \in \{0,\ldots,k-1\}$, and $k$ is the {\em length} of the path.

A subset $Y$ of a digital image $(X,\kappa)$ is
{\em $\kappa$-connected}~\cite{Rosenfeld},
or {\em connected} when $\kappa$
is understood, if for every pair of points $a,b \in Y$ there
exists a $\kappa$-path in $Y$ from $a$ to $b$.

We define
\[ N(X,\kappa, x) = \{ y \in X \, | \, x \adj_{\kappa} y\}.
\]

\begin{definition}
\label{bdDef}
Let $X \subset \Z^n$.
\begin{itemize}
    \item The
{\em boundary of} $X$
{\rm \cite{RosenfeldMAA}} is
\[Bd(X) = \{x \in X \, | \mbox{ there exists } y \in \Z^n \setminus X \mbox{ such that } y \adj_{c_1} x\}.
\]
\item The {\em interior of} $X$
is $Int(X) = X \setminus Bd(X)$.
\end{itemize}
\end{definition}

\subsection{Digitally continuous functions}
Material in this section is quoted or paraphrased
from~\cite{BxFpSets}.

The following generalizes a definition of~\cite{Rosenfeld}.

\begin{definition}
\label{continuous}
{\rm ~\cite{Bx99}}
Let $(X,\kappa)$ and $(Y,\lambda)$ be digital images. 
A function $f: X \rightarrow Y$ is 
{\em $(\kappa,\lambda)$-continuous} if for
every $\kappa$-connected $A \subset X$ we have that
$f(A)$ is a $\lambda$-connected subset of $Y$.
If $(X,\kappa)=(Y,\lambda)$, we say such a function is {\em $\kappa$-continuous},
denoted $f \in C(X,\kappa)$.
$\Box$
\end{definition}

When the adjacency relations are understood, we may simply say that 
$f$ is \emph{continuous}. Continuity can be expressed in terms of 
adjacency of points:
\begin{thm}
{\rm ~\cite{Rosenfeld,Bx99}}
A function $f: (X,\kappa) \to (Y,\lambda)$ is continuous if and only if $x \adj_{\kappa} x'$ in $X$ implies 
$f(x) \adjeq_{\lambda} f(x')$.
\end{thm}

Similar notions are referred to as {\em immersions}, 
{\em gradually varied operators}, and {\em gradually varied mappings}
in~\cite{Chen94,Chen04}.

Composition preserves continuity, in the sense of the following.

\begin{thm}
{\rm \cite{Bx99}}
\label{composition}
Let $(X,\kappa)$, $(Y,\lambda)$, and $(Z,\mu)$ be digital images.
Let $f: X \to Y$ be $(\kappa,\lambda)$-continuous and let
$g: Y \to Z$ be $(\lambda,\mu)$-continuous. Then
$g \circ f: X \to Z$ is $(\kappa,\mu)$-continuous.
\end{thm}

Given $X = \Pi_{i=1}^v X_i$, we denote throughout this paper the projection
onto the $i^{th}$ factor by $p_i$; i.e., $p_i: X \to X_i$ is defined by
$p_i(x_1,\ldots,x_v) = x_i$, where $x_j \in X_j$.

Given a function $f: X \to X$, we say $x \in X$ is
a {\em fixed point of} $f$ if $f(x)=x$. The set of
points $\{x \in X \, | \, f(x)=x\}$ we denote as
$\Fix(f)$.

We use the notation $\id_X$ to denote the
{\em identity function}: $\id_X: X \to X$ is
the function $\id_X(x) = x$ for all $x \in X$.

\begin{definition}
{\rm \cite{BxFpSets}}
\label{freezeDef}
Let $(X,\kappa)$ be a digital image. We say
$A \subset X$ is a {\em freezing set for $X$}
if given $f \in C(X,\kappa)$,
$A \subset \Fix(f)$ implies $f=\id_X$.
\end{definition}

\subsection{Digital disks}
Let $\kappa \in \{c_1,c_2\}$. We say a $\kappa$-connected 
set $S=\{x_i\}_{i=1}^n \subset \Z^2$ is a
{\em (digital) line segment} if the members of $S$ are collinear.

\begin{remark}
A digital line segment must be vertical, horizontal, or have
slope of $\pm 1$. We say a segment with slope of $\pm 1$ is
{\em slanted}.
\end{remark}

A {\em (digital) $\kappa$-closed curve} is a
path $S=\{s_i\}_{i=0}^{m-1}$ such that $s_0=s_{m-1}$,
and $|i - j| < m-1$ 
implies $s_i \neq s_j$. If 
$s_i \adj_{\kappa} s_j$ implies 
$|i - j| \mod m = 1$, $S$ is a {\em (digital) 
$\kappa$-simple closed curve}.
For a simple closed curve $S \subset \Z^2$ we generally assume
\begin{itemize}
    \item $m \ge 8$ if $\kappa = c_1$, and
    \item $m \ge 4$ if $\kappa = c_2$.
\end{itemize}
These requirements are necessary for the Jordan Curve
Theorem of digital topology, below, as a
$c_1$-simple closed curve in $\Z^2$ needs at least 8 points to
have a nonempty finite complementary $c_2$-component,
and a $c_2$-simple closed curve in $\Z^2$ needs at least 4 points to
have a nonempty finite complementary $c_1$-component.
Examples in~\cite{RosenfeldMAA} show why it is
desirable to consider $S$ and $\Z^2 \setminus S$
with different adjacencies.

\begin{thm}
{\rm \cite{RosenfeldMAA}}
{\em (Jordan Curve Theorem for digital topology)}
Let $\{\kappa, \kappa'\} = \{c_1, c_2\}$.
Let $S \subset \Z^2$ be a simple closed 
$\kappa$-curve such that $S$ has at least 8 points if
$\kappa = c_1$ and such that $S$ has at least 
4 points if $\kappa = c_2$. Then
$\Z^2 \setminus S$ has exactly 2 $\kappa'$-connected
components.
\end{thm}

One of the $\kappa'$-components of 
$\Z^2 \setminus S$ is finite and the other is infinite. This 
suggests the following.
\begin{definition}
\label{diskDef}
Let $S \subset \Z^2$ be a $c_2$-closed curve such that
$\Z^2 \setminus S$ has two $c_1$-components, one finite and the
other infinite. The union $D$ of $S$ and the finite $c_1$-component 
of $\Z^2 \setminus S$ is a {\em (digital) disk}. $S$ is
a {\em bounding curve} of $D$. The finite component $c_1$-component 
of $\Z^2 \setminus S$ is the {\em interior of} $S$.
\end{definition}

We will consider a given disk $D$ with either the $c_1$ or the
$c_2$ adjacency.
Notes:
\begin{itemize}
    \item If $D$ is a digital disk determined as above by a bounding 
    $c_2$-closed curve $S$, then $(S,c_1)$ can be 
    disconnected. See Figure~\ref{fig:diamond}.
    \item There may be more than one closed curve $S$
          bounding a given disk~$D$. See Figure~\ref{fig:2sccBdry}.
          Since we are interested in finding {\em minimal}
          freezing sets and since it turns out we often compute these
          from bounding curves, we will generally prefer those that are
          components of $Bd(D)$ so that we
          can use Theorem~\ref{bdFreezes}; or
          those that are
          {\em minimal} bounding curves. A bounding curve~$S$
          for a disk $D$ is {\em minimal} if there is no
          bounding curve $S'$ for $D$ such that
          $\#S' < \#S$.
    \item In particular, a bounding
          curve need not be equal to $Bd(D)$.
          E.g., in the disk~$D$
          shown in Figure~\ref{fig:2sccBdry}(i), $(2,2)$ is a point
          of the bounding curve; however, all of the points
          $c_1$-adjacent to $(2,2)$ are members of~$D$, so
          by Definition~\ref{bdDef}, $(2,2) \not \in Bd(D)$.
          Thus, a bounding curve for $D$ need not be contained
          in $Bd(D)$.
    \item In Definition~\ref{diskDef}, we use $c_2$ adjacency for
          $S$ and we do not require $S$ to be simple. 
          Figure~\ref{fig:2sccBdry} shows why these seem
          appropriate.
          \begin{itemize}
              \item The use of $c_2$ adjacency allows slanted
              segments in bounding curves and makes possible
          a bounding curve in subfigure~(ii) with fewer points
          than the bounding curve in subfigure~(i) in which
          adjacent pairs of the bounding curve are restricted
          to $c_1$ adjacency.
          \item Neither of the bounding curves shown in
                Figure~\ref{fig:2sccBdry} is a $c_2$-simple closed
                curve. E.g., non-consecutive points of each of
                the bounding curves,
                $(0,1)$ and $(1,0)$, are $c_2$-adjacent. The
                bounding curve shown in 
                Figure~\ref{fig:2sccBdry}(ii) is clearly also not a
                $c_1$-simple closed curve.
          \end{itemize}
    \item A closed curve that is not simple may be the boundary
          of a digital image that is not a disk. This is illustrated
          in Figure~\ref{fig:notDisk}.
\end{itemize}

\begin{figure}
    \centering
    \includegraphics[height=1.5in]{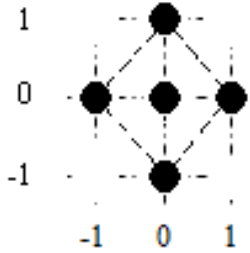}
    \caption{The $c_1$-disk
    $D = \{(x,y) \in \Z^2 \, | \, |x| + |y| < 2\}$.
    The bounding curve
    $S = \{(x,y) \in \Z^2 \, | \,
    |x| + |y| =1\} = D \setminus \{(0,0)\}$
    is not $c_1$-connected.
    }
    \label{fig:diamond}
\end{figure}

\begin{figure}
    \centering
    \includegraphics[height=2in]{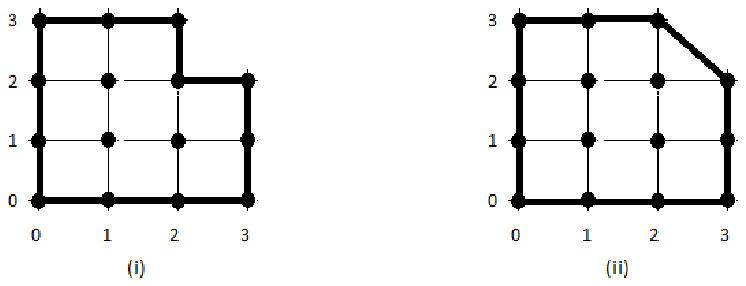}
    \caption{Two views of $D = [0,3]_{\Z} \setminus \{(3,3)\}$, 
    which can be regarded as a $c_1$-disk with either of the
    closed curves shown in dark as a bounding curve. \newline
    (i) The dark line segments show a $c_1$-simple closed curve $S$
    that is a bounding curve for~$D$. \newline
    (ii) The dark line segments show a $c_2$-closed curve $S$
    that is a minimal bounding curve for~$D$. \newline
    Since Theorems~\ref{convexDiskThm} and~\ref{convexDiskThmC2}
    suggest computing minimal freezing sets from bounding curves,
    use of a minimal bounding curve is sometimes preferred. Note
    without the restriction of minimality, were the bounding
    curve in (i) considered,
    Theorem~\ref{convexDiskThm} could incorrectly suggest 
    $(2,2)$ as a point of the minimal freezing set
    for $(D,c_1)$ even though by Definition~\ref{bdDef},
    $(2,2) \not \in Bd(D)$; the minimal
    bounding curve in (ii) does not lead to this incorrect
    suggestion.
    }
    \label{fig:2sccBdry}
\end{figure}

\begin{figure}
    \centering
    \includegraphics{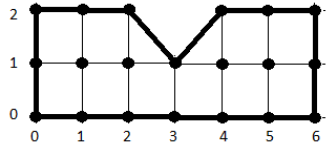}
    \caption{$D = [0,6]_{\Z} \times [0,2]_{\Z} \setminus \{(3,2)\}$
     shown with a bounding curve $S$ in dark segments. $D$ is
     not a disk with either the $c_1$ or the $c_2$ adjacency,
     since with either of these adjacencies,
     $\Z^2 \setminus S$ has two bounded components,
     $\{(1,1), (2,1)\}$ and $\{(4,1), (5,1)\}$.
     }
    \label{fig:notDisk}
\end{figure}

A set $X$ in a Euclidean space $\R^n$ is
{\em convex} if for every pair of distinct
points $x,y \in X$, the line segment
$\overline{xy}$ from $x$ to $y$ is contained in $X$.
The {\em convex hull of} $Y \subset \R^n$,
denoted $hull(Y)$, is the
smallest convex subset of $\R^n$ that contains~$Y$.
If $Y \subset \R^2$ is a finite set, then
$hull(Y)$ is a single point if $Y$ is a singleton;
a line segment if $Y$ has at least 2 members and all are
collinear; otherwise, $hull(Y)$ is a polygonal disk,
and the endpoints of the edges of $hull(Y)$ are its {\em vertices}.

A digital version of convexity can be stated
for subsets of the digital plane~$\Z^2$ as follows.
A finite set $Y \subset \Z^2$ is 
{\em (digitally) convex} if either
\begin{itemize}
    \item $Y$ is a single point, or
    \item $Y$ is a digital line segment, or
    \item $Y$ is a digital disk with a bounding curve $S$
          such that the endpoints of the maximal line segments
          of~$S$ are the vertices of $hull(Y) \subset \R^2$.
\end{itemize}

Let $s_1$ and $s_2$ be sides of a digital disk
$X \subset \Z^2$, i.e., maximal digital line segments
in a bounding curve $S$ of $X$, such that 
$s_1 \cap s_2 = \{p\} \subset X$.
The {\em interior angle of $X$ at $p$} is the
angle formed by $s_1$, $s_2$, and $Int(X)$.

\begin{remark}
Let $(X,\kappa)$ be a digital disk in $\Z^2$, 
$\kappa \in \{c_1,c_2\}$. Let $s_1$ and $s_2$ be sides of
$X$ such that $s_1 \cap s_2 = \{p\} \subset X$. Then
the interior angle of $X$ at $p$ is well defined.
\end{remark}

\begin{proof}
If there exists
$q \in X \setminus (s_1 \cup s_2)$ such that
$q \adj_{c_2} p$, then
the interior angle of $X$ at $p$ is the angle obtained by
rotating $s_1$ about $p$ through $q$ to reach $s_2$.

Otherwise, the angles formed by $s_1$ and $s_2$ measure
45$^\circ$ ($\pi / 4$ radians) and 315$^\circ$
($7 \pi /4$ radians). The latter has a point 
$q \in \Z^2 \setminus X$ such
that $q \adj_{c_2} p$. Therefore, the 45$^\circ$
angle determined by $s_1$ and $s_2$ is the interior angle
of $X$ at $p$.
\end{proof}

\subsection{Tools for determining fixed point sets}
The following assertions will be useful in
determining fixed point and freezing sets.

\begin{prop}
\label{uniqueShortestProp}
{\rm (Corollary~8.4 of~\cite{bs19a})}
Let $(X,\kappa)$ be a digital image and
$f \in C(X,\kappa)$. Suppose
$x,x' \in \Fix(f)$ are such that
there is a unique shortest
$\kappa$-path $P$ in~$X$ from $x$ 
to $x'$. Then $P \subset \Fix(f)$.
\end{prop}

Lemma~\ref{cuPulling} below is in the spirit of ``pulling" as
introduced in~\cite{hmps}.
We quote~\cite{BxFpSets}:
\begin{quote}
    The following assertion can
be interpreted to say that
in a $c_u$-adjacency,
a continuous function that moves
a point~$p$ also [pulls along]
a point that is ``behind"
$p$. E.g., in $\Z^2$, if $q$ and $q'$ are
$c_1$- or $c_2$-adjacent with $q$
left, right, above, or below $q'$, and a
continuous function $f$ moves $q$ to the left,
right, higher, or lower, respectively, then
$f$ also moves $q'$ to the left,
right, higher, or lower, respectively.
\end{quote}

\begin{lem}
\label{cuPulling}
{\rm \cite{BxFpSets}}
Let $(X,c_u)\subset \Z^n$ be a digital image, 
$1 \le u \le n$. Let $q, q' \in X$ be such that
$q \adj_{c_u} q'$.
Let $f \in C(X,c_u)$.
\begin{enumerate}
    \item If $p_i(f(q)) > p_i(q) > p_i(q')$
          then $p_i(f(q')) > p_i(q')$.
    \item If $p_i(f(q)) < p_i(q) < p_i(q')$
          then $p_i(f(q')) < p_i(q')$.
\end{enumerate}
\end{lem}

\begin{figure}
    \centering
    \includegraphics[height=1in]{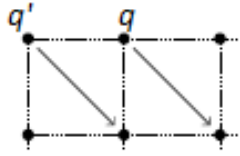}
    \caption{Illustration of Lemma~\ref{cuPulling}. Arrows show
    the images of $q,q'$ under $f \in C(X,c_2)$. Since
    $f(q)$ is to the right of $q$ and $q' \adj_{c_1,c_2} q$ with
    $q'$ to the left of $q$, $f$ pulls $q'$ to the right so that
    $f(q')$ is to the right of $q'$.
    }
    \label{fig:pull}
\end{figure}

Figure~\ref{fig:pull} illustrates Lemma~\ref{cuPulling}.

\begin{thm}
\label{bdFreezes}
{\rm \cite{BxFpSets}}
Let $X \subset \Z^n$ be finite. Then 
for $1 \le u \le n$, $Bd(X)$ is 
a freezing set for $(X,c_u)$.
\end{thm}

\begin{thm}
\label{bdCurveFreezeSet}
Let $D$ be a digital disk in $\Z^2$. Let
$S$ be a bounding curve for $D$. Then $S$ is
a freezing set for $(D,c_1)$ and for $(D,c_2)$.
\end{thm}

\begin{proof}
This is like the proof of 
Theorem~\ref{bdFreezes} in~\cite{BxFpSets}.
Let $\kappa \in \{c_1,c_2\}$. Let
$f \in C(D,\kappa)$ such that
$S \in \Fix(f)$. Suppose there exists
$x \in D$ such that $f(x) \neq x$.
Then $x$ lies on a horizontal segment
$\overline{ab}$ and on a vertical segment
$\overline{cd}$ such that 
$\{a,b,c,d\} \subset S$, $p_1(a) < p_1(b)$,
and $p_2(c) < p_2(d)$.
\begin{itemize}
    \item If $p_1(f(x)) > p_1(x)$ then by
          Lemma~\ref{cuPulling}, 
          $p_1(f(a)) > p_1(a)$, contrary to
          $a \in S \subset \Fix(f)$.
    \item If $p_1(f(x)) < p_1(x)$ then by
          Lemma~\ref{cuPulling}, 
          $p_1(f(b)) < p_1(b)$, contrary to
          $b \in S \subset \Fix(f)$.
    \item If $p_2(f(x)) > p_2(x)$ then by
          Lemma~\ref{cuPulling}, 
          $p_1(f(c)) > p_1(c)$, contrary to
          $c \in S \subset \Fix(f)$.
    \item If $p_2(f(x)) < p_2(x)$ then by
          Lemma~\ref{cuPulling}, 
          $p_1(f(d)) < p_1(d)$, contrary to
          $d \in S \subset \Fix(f)$.
\end{itemize}
In all cases, we have a contradiction brought
on by assuming $x \not \in \Fix(f)$. Therefore,
$f = \id_D$, so $S$ is a freezing set for
$(D,\kappa)$.
\end{proof}

\section{$c_1$-Freezing sets for disks in $\Z^2$}
The following can be interpreted as stating that
the set of ``corner points" form a freezing set for a
digital cube with the $c_1$ adjacency.

\begin{thm}
\label{corners-min}
{\rm \cite{BxFpSets}}
Let $X = \Pi_{i=1}^n [0,m_i]_{\Z}$.
Let $A = \Pi_{i=1}^n \{0,m_i\}$.
Then $A$ is a freezing set for $(X,c_1)$; minimal for $n \in \{1,2\}$.
\end{thm}

\begin{remark}
Example~5.16 of~{\rm \cite{BxFpSets}} shows that the freezing set
of Theorem~\ref{corners-min} need not be minimal for
$n=3$.
\end{remark}

The argument used to prove Theorem~\ref{corners-min}
may lead one to ask if this theorem can be generalized
as follows:
\begin{quote}
    Given a digital disk $D \subset \Z^2$ such
    that all of the maximal segments of a bounding curve of $D$
    are horizontal or vertical, is the set
    of the endpoints of the maximal segments
    of a bounding simple closed curve $S$ a minimal 
    freezing set for $(D,c_1)$?
\end{quote}

The following provides a negative answer to this
question.

\begin{exl}
\label{axesParallelCounterexl}
Let $D = [0,3]_{\Z} \times [0,6]_{\Z}
         \setminus \{(3,3)\})$.
Then 
\[A = \{(0,0), (3,0), (3,2), (3,4), (3,6), (0,6)\} \]
(see Figure~\ref{fig:nonConvHorzVertBd})
is a minimal freezing set for $(D,c_1)$.
Note $(2,2)$ and $(2,4)$ are endpoints of maximal horizontal
and vertical bounding segments of $D$ and are 
not members of $A$. While $(2,2)$ and $(2,4)$ are members of a
bounding curve for~$D$, they are not members of a minimal bounding
curve, which includes edges from $(3,4)$ to $(2,3)$ and from
$(2,3)$ to $(3,2)$.
\end{exl}

    \begin{figure}
        \centering
        \includegraphics[height=2.5in]{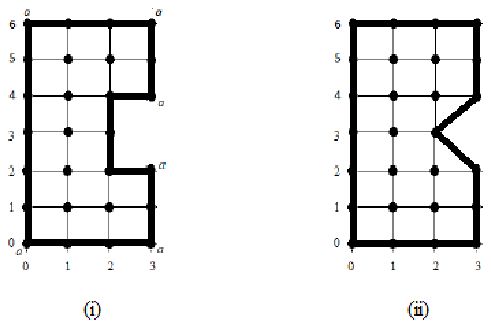}
        \caption{There are distinct boundary
        curves for the disk $D$ that contain the
        horizontal segments from $(0,0)$ to $(3,0)$ and from
        $(0,6)$ to $(3,6)$; and vertical segments from
        $(0,0)$ to $(0,6)$, from $(3,0)$ to $(3,2)$, and from
        $(3,4)$ to $(3,6)$. \newline
        (i) We can complete a boundary curve by using the horizontal segments from $(2,2)$ to $(3,2)$ and from $(2,4)$ to $(3,4)$
        and the vertical segment from $(2,2)$ to $(2,4)$, as shown in dark. This lets
        us view $D$ as a disk with horizontal and vertical
        sides. Members of the minimal freezing
        set $A$ for $(D,c_1)$, determined in
        Example~\ref{axesParallelCounterexl},
        are marked ``{\em a}". Note
        $\{(2,2), (2,4)\} \cap A = \emptyset$. 
        $(2,2)$ and $(2,4)$ are endpoints of a 
        maximal horizontal segment of a
        bounding curve, but not of
        the minimal bounding curve~$S$;
        the latter is shown in (ii). Indeed, by Definition~\ref{diskDef},
        $\{(2,2), (2,4)\} \subset Int(D)$.
        \newline
        (ii) Alternately, we can complete a boundary curve by using
        the slanted line segments from $(2,3)$ to $(3,4)$ and
        from $(2,3)$ to $(3,2)$. This is a minimal boundary curve
        $S$ that lets us view $D$ as in Example~\ref{nonConvC2Exl}.
        A minimal freezing set for $(D,c_2)$ is 
        $S \setminus \{(2,3)\}$.
        }
        \label{fig:nonConvHorzVertBd}
    \end{figure}

\begin{proof}
Let $f \in C(D,c_1)$ such that $A \subset \Fix(f)$.
It follows from Proposition~\ref{uniqueShortestProp}
that the vertical segments
$\{0\} \times [0,6]_{\Z}$, 
$\{3\} \times [0,2]_{\Z}$, and $\{3\} \times [4,6]_{\Z}$,
the horizontal segments
$[0,3]_{\Z} \times \{0\}$ and
$[0,3]_{\Z} \times \{6\}$, and the
path 
\[\{(3,2), (2,2), (2,3), (2,4), (3,4) \}
\]
are all
subsets of $\Fix(f)$. Since 
the union of these paths is a bounding
curve $S$ for $D$, we have 
$S \subset \Fix(f)$. That $A$ is a freezing set
follows from Theorem~\ref{bdCurveFreezeSet}.

To show $A$ is a minimal freezing set, we observe that
for each $p \in A$ there is a function $f_p: D \to D$ defined by
\[ f_p(x) = \left \{ \begin{array}{ll}
     (1,1) & \mbox{if } x = p= (0,0); \\
     (2,1) & \mbox{if } x = p \in \{(3,0),(3,2)\}; \\
     (2,5) & \mbox{if } x = p \in \{(3,4), (3,6)\}; \\
     (1,5) & \mbox{if } x = p= (0,6); \\
    x & \mbox{if } x \neq p.  \\
 \end{array} \right .
\]
It is easily seen that each 
$f_p \in C(D, c_1)$, with 
$\Fix(f_p) = D \setminus \{p\}$.
It follows that $A \setminus \{p\}$ is not a freezing set
for $(D,c_1)$, so $A$ is a minimal freezing set.
\end{proof}

\begin{definition}
\label{thickness}
Let $X \subset \Z^2$ be a digital disk. We say $X$ is
{\em thick} if the following are satisfied. For some bounding
curve $S$ of $X$,
\begin{itemize}
    \item for every slanted segment~$S$ of $Bd(X)$,
 if  $p \in S$ is not an endpoint of  $S$, 
then there exists $c \in X$ such that 
(see Figure~\ref{fig:innerBdPt})
\begin{equation}
    \label{slantSegProp}
   c \adj_{c_2} p \not \adj_{c_1} c,
\end{equation}
and
\item if $p$ is the vertex of a 135$^\circ$ ($3 \pi / 4$
      radians) interior angle $\theta$ of $S$,
      there exist $b,b' \in X$
      such that $b$ and $b'$ are in the interior of $\theta$ and
      (see Figure~\ref{fig:degrees135c1})
      \[ b \adj_{c_2} p \not \adj_{c_1} b~~~ \mbox{ and }~~~ 
      b' \adj_{c_1} p.
       \]
\end{itemize}
\end{definition}

    \begin{figure}
        \centering
        \includegraphics[height=2in]{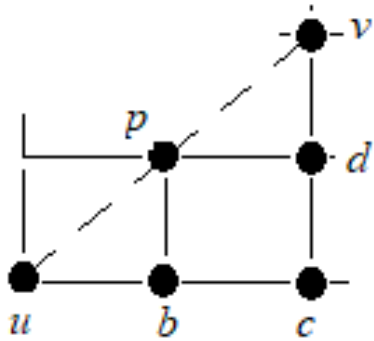}
        \caption{$p \in \overline{uv}$ in a bounding curve,
        with $\overline{uv}$ slanted.
        Note $u \not \adj_{c_1} p \not \adj_{c_1} v$,
        $p \adj_{c_2} c \not \adj_{c_1} p$,
        $\{p,c\} \subset N(\Z^2,c_1,b) \cap N(\Z^2,c_1,d)$. If
        $X$ is thick then $c \in X$.
        (Not meant to be understood as showing all of $X$.)}
        \label{fig:innerBdPt}
    \end{figure}

         \begin{figure}
        \centering
        \includegraphics[height=2in]{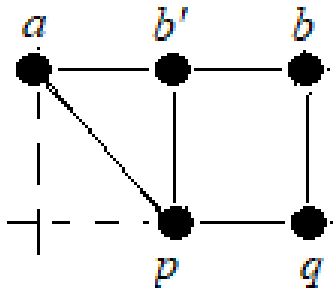}
        \caption{$\angle apq$ is an angle of
        135$^ \circ$ degrees ($3 \pi /4$ radians)
        of a bounding curve of $X$ at $p$, with
        $\overline{ap} \cup \overline{pq}$
            a subset of the bounding curve. If
            $X$ is thick then $b,b' \in X$. (Not meant to
        be understood as showing all of $X$.)
        }
        \label{fig:degrees135c1}
        \end{figure}

Examples of digital images that fail to be thick are shown
in Figure~\ref{fig:notThick}.

\begin{figure}
    \includegraphics{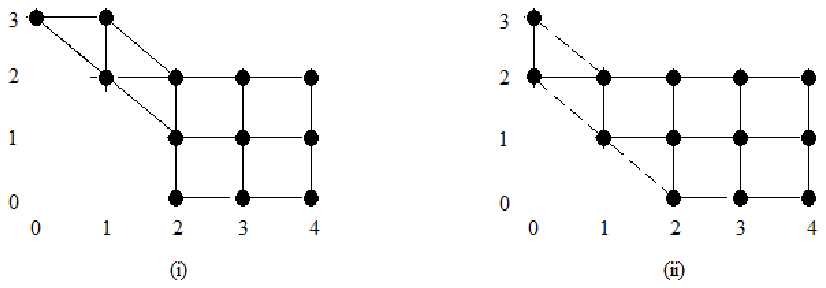}
    \caption{Two digital disks that are not thick. \newline (i) $(1,2)$
       is a non-endpoint of a slanted boundary segment for which
       there is no point corresponding to $c$ of 
       Figure~\ref{fig:innerBdPt}. \newline (ii) $(0,2)$ is the
       vertex of a 135$^\circ$ interior angle of a bounding curve
       for which there is no point corresponding to $b$ of
       Figure~\ref{fig:degrees135c1}.
       }
    \label{fig:notThick}
\end{figure}

The following expands on the
dimension~2 case of Theorem~\ref{corners-min} to give a subset
of $Bd(X)$ that is a freezing set.

\begin{thm}
\label{convexDiskThm}
Let $X$ be a finite digital image in~$\Z^2$ such that
$Bd(X)=\bigcup_{i=1}^n S_i$ is the disjoint union of 
$c_2$-closed curves $S_i$.
Let $A_1$ be the set of points $x \in Bd(X)$ such that
$x$ is an endpoint of a maximal horizontal or a
maximal vertical edge of some $S_i$. Let $A_2$ 
be the union of slant line segments in $Bd(X)$.
Then $A = A_1 \cup A_2$ is a freezing set
for $(X,c_1)$.
\end{thm}

\begin{proof}
Let $x,x'$ be distinct members of $A_1$ that are
endpoints of the same maximal horizontal or vertical
edge $E$ in some $S_i$. Then $E$ contains the unique
shortest $c_1$-path in $X$ from $x$ to $x'$. By
Proposition~\ref{uniqueShortestProp}, if
$f \in C(X,c_1)$ and $\{x,x'\} \subset \Fix(f)$,
then $E \subset \Fix(f)$. By hypothesis
we also have that
$A_2 \subset \Fix(f)$, so $S_i \subset \Fix(f)$. Therefore,
$Bd(X) \subset \Fix(f)$.
By Theorem~\ref{bdFreezes}, $f = \id_X$. Thus $A$ is a 
freezing set for $(X,c_1)$.
\end{proof}

\begin{remark}
The set $A$ of Theorem~\ref{convexDiskThm}
need not be minimal. This is shown in 
Example~\ref{axesParallelCounterexl}, 
where $(2,3)$, as a member of a 
slanted edge of a minimal bounding curve
(see Figure~\ref{fig:nonConvHorzVertBd}), is a
member of the set $A$ of Theorem~\ref{convexDiskThm}, but
is not a member of the minimal freezing set.
\end{remark}

\begin{thm}
\label{convDiskThmActual}
Let $X$ be a thick convex disk with a
    bounding curve $S$,
    Let $A_1$ be the set of points $x \in S$ such that
$x$ is an endpoint of a maximal horizontal or a
maximal vertical edge of $S$. Let $A_2$ 
be the union of slant line segments in $S$.
Then $A = A_1 \cup A_2$ is a minimal 
freezing set for $(X,c_1)$ (see  
Figure~\ref{fig:convexRlts}(ii)).
\end{thm}

\begin{figure}
    \centering
    \includegraphics{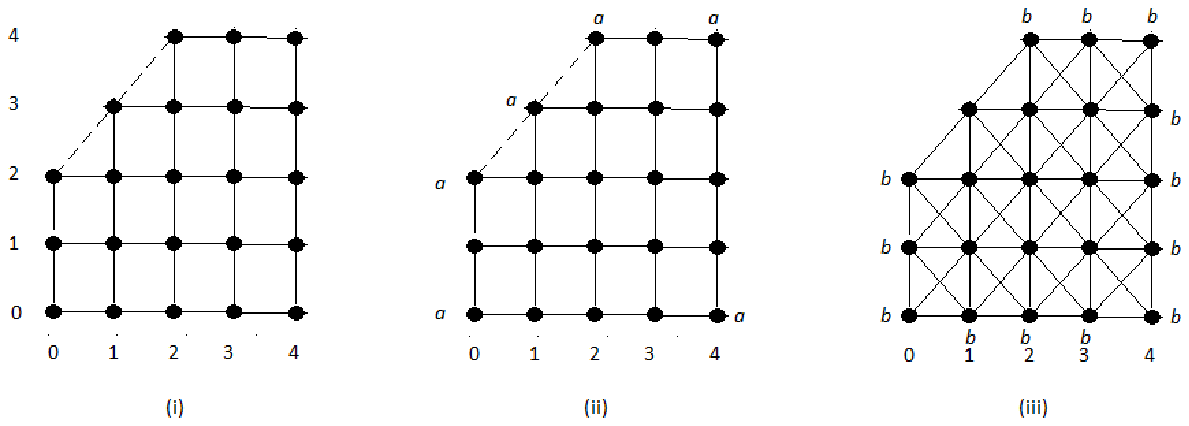}
    \caption{The convex disk 
    $D = [0,4]_{\Z}^2 \setminus \{(0,3),(0,4),(1,4)\}$. The dashed
    segment from $(0,2)$ to $(2,4)$ shown in (i) and (ii) indicates
    part of the bounding curve and not $c_1$-adjacencies. \newline
    (i) $D$ with a $c_2$ bounding curve. \newline
    (ii) $(D,c_1)$ with members of a minimal freezing set $A$
    marked ``{\em a}" - these are the endpoints of the maximal 
    horizontal and vertical segments of the bounding curve, 
    and all points of the slanted segment of the bounding curve,
    per Theorem~\ref{convexDiskThm}. \newline
    (iii) $(D,c_2)$ with members of a minimal freezing set 
    $B$ marked ``{\em b}" - these are the endpoints of the maximal
    slanted edge and all the points of the horizontal and vertical
    edges of the bounding curve,
    per Theorem~\ref{convexDiskThmC2}.
    }
    \label{fig:convexRlts}
\end{figure}

\begin{proof}
That $A$ is a freezing set follows as in the
proof of Theorem~\ref{convexDiskThm}.
To show $A$ is minimal, we must show that if we
remove a point $p$ from $A$, the remaining set
$A \setminus \{p\}$ is not a freezing set. 

We start by considering $p \in A_1$. 
Since $X$ is convex, the interior angle of $S$ at $p$ must be
$45^{\circ}$ ($\pi/4$ radians),
$90^{\circ}$ ($\pi/2$ radians), or $135^{\circ}$ 
($3\pi/4$ radians).
 \begin{figure}
        \centering
        \includegraphics{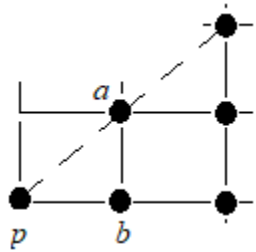}
        \caption{$\angle apb$ is a
        $45^{\circ}$ ($\pi/4$ radians) interior
        angle of a bounding curve at $p \in A_1$.
        (Not meant to
        be understood as showing all of $X$.)}
        \label{fig:degrees45}
    \end{figure}
\begin{itemize}
    \item Suppose the interior angle of $S$ at $p$ is
   $45^{\circ}$ ($\pi/4$ radians). Let $b$ be a
    point of $S$ that is $c_1$-adjacent 
    to $p$ on the horizontal or vertical edge of this angle (see Figure~\ref{fig:degrees45}). Then the function
    $f: X \to X$ defined by
    \[ f(x) = \left \{ \begin{array}{ll}
        x & \mbox{if } x \neq p; \\
        b & \mbox{if } x = p,
    \end{array}
      \right .
    \]
    satisfies $f \in C(X,c_2)$, with 
    $\Fix(f) = X \setminus \{p\}$. Thus
    $X \setminus \{p\}$ is not a freezing set for $(X,c_2)$.
    \item Suppose the interior angle of $S$ at $p$ is
    $90^{\circ}$ ($\pi/2$ radians). Let $a,b$ be the
    points of $S$ that are $c_1$-adjacent 
    to $p$ on the horizontal and vertical edges of this angle
    and let $q$ be the point of
    $Int(X)$ that is $c_1$-adjacent to each of
    $a$ and $b$ (see Figure~\ref{fig:degrees90a}).
    \begin{figure}
        \centering
        \includegraphics{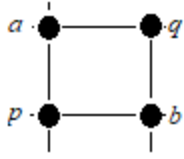}
        \caption{$\angle apb$ is a
        $90^{\circ}$ ($\pi/2$ radians)
        angle of a bounding curve of $X$ at $p \in A_1$, with
        horizontal and vertical sides.
        $q \in Int(X)$. (Not meant to
        be understood as showing all of $X$.)}
        \label{fig:degrees90a}
    \end{figure}
    Then the function $f: X \to X$ defined by
    \[ f(x) = \left \{ \begin{array}{ll}
        q & \mbox{if } x=p;  \\
        x & \mbox{if } x \neq p
    \end{array}
    \right .
    \]
    belongs to $C(X,c_1)$. Also,
    $\Fix(f) = X \setminus \{p\}$. Thus,
    $A \setminus \{p\}$ is not a freezing set
    for $(X,c_1)$.
    \item Suppose the interior angle of $S$ at $p$ is
    $135^{\circ}$ ($3\pi/4$ radians). 
        Let $a,q \in S$ be such that
    $a$ and $q$ are the members of this angle that are
    $c_2$-adjacent to $p$, where $\overline{ap}$ is slanted
    and $\overline{pq}$ is horizontal or vertical. Since $X$
    is thick, Definition~\ref{thickness} yields that
    there exists $b \in X$ such that
    $b \adj_{c_2} p$
    (as in Figure~\ref{fig:degrees135c1}). 
    Then the function $f: X \to X$ defined by
    \[ f(x) = \left \{ \begin{array}{ll}
        b & \mbox{if } x=p;  \\
        x & \mbox{if } x \neq p
    \end{array}
    \right .
    \]
        belongs to $C(X,c_1)$ (note, as shown
        in Figure~\ref{fig:degrees135c1}.
        $p \not \adj_{c_1} a$). Also,
    $\Fix(f) = X \setminus \{p\}$. Thus,
    $A \setminus \{p\}$ is not a freezing set
    for $(X,c_1)$.
\end{itemize}
Thus we have shown that if $p \in A_1$ then
$A \setminus \{p\}$ is not a freezing set
for $(X,c_1)$.

Now we wish to show if $p \in A_2$ then
$A \setminus \{p\}$ is not a freezing set
for $(X,c_1)$. Let $s$ be a slanted segment of $Bd(X)$
containing $p$.

If $p$ is not an endpoint of $s$, then
from the assumption~(\ref{slantSegProp})
there exist $b,c,d \in X$ such that 
$p \adj_{c_2} c$, $p \not \adj_{c_1} c$, and
$b \adj_{c_1} c \adj_{c_1} d$ 
(see Figure~\ref{fig:innerBdPt}). Then the function
    $f: X \to X$ defined by
    \[ f(x) = \left \{ \begin{array}{ll}
        c & \mbox{if } x=p;  \\
        x & \mbox{if } x \neq p
    \end{array}
    \right .
    \]
    belongs to $C(X,c_1)$. Also,
    $\Fix(f) = X \setminus \{p\}$. Thus,
    $A \setminus \{p\}$ is not a freezing set
    for $(X,c_1)$.

If $p$ is an endpoint of $s$, let $s'$ be the
other maximal segment of $Bd(X)$ for which $p$ is
an endpoint. If $s'$ is horizontal or vertical,
then $p \in A_1$, hence, as discussed above,
$A \setminus \{p\}$ is not a freezing set for $(X,c_1)$.
Therefore, we assume $s'$ is slanted.
Since $X$ is convex and both $s$ and $s'$ are slanted,
the interior angle of $S$
at $p$ must be $90^\circ$ ($\pi/2$ radians). There
exists $q \in Int(X)$ such that $q \adj_{c_1} p$
(see Figure~\ref{fig:degrees90b}).
   \begin{figure}
        \centering
        \includegraphics[height=2in]{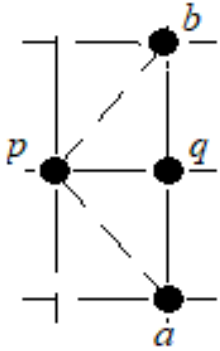}
        \caption{$\angle apb$ is a
        $90^\circ$ ($\pi/2$ radians) angle
         between slanted segments of a bounding curve, with
         $q \in Int(X)$. (Not meant to
        be understood as showing all of $X$).}
        \label{fig:degrees90b}
    \end{figure}
Then the function
    $f: X \to X$ defined by
    \[ f(x) = \left \{ \begin{array}{ll}
        q & \mbox{if } x=p;  \\
        x & \mbox{if } x \neq p
    \end{array}
    \right .
    \]
    belongs to $C(X,c_1)$. Also,
    $\Fix(f) = X\setminus \{p\}$. Thus,
    $A \setminus \{p\}$ is not a freezing set
    for $(X,c_1)$.
\end{proof}

\section{$c_2$-Freezing sets for disks in $\Z^2$}
For disks in $\Z^2$, we obtain results
for the $c_2$ adjacency
that are dual to those obtained for the $c_1$ 
adjacency in the previous section.

As was true of the $c_1$ adjacency and Theorem~\ref{convexDiskThm}, we see, by
comparing Example~\ref{nonConvC2Exl} and
Theorem~\ref{convDiskThmC2Actual} below, that
with $c_2$ adjacency, 
convexity can affect determination of a minimal
freezing set for a digital image in~$\Z^2$.

\begin{exl}
\label{nonConvC2Exl}
Let $D = [0,3]_{\Z} \times [0,6]_{\Z}
         \setminus \{(3,3)\})$.
(This is the set used in 
Example~\ref{axesParallelCounterexl}.
See Figure~\ref{fig:nonConvHorzVertBd}.)
Let
\[ B = Bd(D) \setminus \{(2,3)\}.
\]
Then $B$ is a minimal freezing set for $(D,c_2)$.
\end{exl}

\begin{proof}
Let $f \in C(D,c_2)$ be such that 
\begin{equation}
\label{fFixesB}
f|_B = \id_B.
\end{equation}
Let $p=(2,3)$, $q=(3,2) \in B$, 
$s=(3,4) \in B$. Note the following:
\begin{itemize}
    \item If $p_1(f(p)) > p_1(p)$ then by Lemma~\ref{cuPulling},
          $p_1(f(1,3) > 1$ and therefore
          $p_1(f(0,3)) > 0$, contrary to~(\ref{fFixesB}).
    \item If $p_1(f(p)) < p_1(p)$ then by Lemma~\ref{cuPulling},
          $p_1(f(q)) < 3$, contrary to~(\ref{fFixesB}).
    \item If $p_2(f(p)) > p_2(p)$ then by Lemma~\ref{cuPulling},
          $p_2(f(q)) > 2$, contrary to~(\ref{fFixesB}).
    \item If $p_2(f(p)) < p_2(p)$ then by Lemma~\ref{cuPulling},
          $p_1(f(s)) < 4$, contrary to~(\ref{fFixesB}).
\end{itemize}
It follows that
$p \in \Fix(f)$. Since $B \cup \{p\} = Bd(D)$,
it follows from Theorem~\ref{bdFreezes} that
$Bd(D) \subset \Fix(f)$. By Theorem~\ref{bdFreezes},
$f = \id_D$. This establishes that $B$ is a
freezing set.

To show $B$ is minimal, for
$b \in B$ let $f_b: D \to D$ be the function
\[ f_b(x) = \left \{ \begin{array}{ll}
   (1,1)  & \mbox{if } x= b=(0,0);  \\
   (i, 1) & \mbox{if } x=b=(i,0) \mbox{ for } i \in \{1,2\}; \\
   (1, j) & \mbox{if } x=b= (0,j) \mbox{ for } 1 \le j \le 5; \\
   (1,5) & \mbox{if } x=b= (0,6); \\
   (i,5) & \mbox{if } x=b= (i,6) \mbox{ for } i \in \{1,2\}; \\
   (2,5) & \mbox{if } x=b= (3,6); \\
   (2,j) & \mbox{if } x=b= (3,j) \mbox{ for } j \in \{1,2,4,5\};\\
   x & \mbox{if } x \neq b.
   \end{array}
\right .
\]
Then $f_b \in C(D,c_2)$ (this is easily seen from
Figure~\ref{fig:nonConvHorzVertBd}), 
and $\Fix(f_b) = D \setminus \{b\}$. Therefore, $B \setminus \{b\}$
is not a freezing set for $(D,c_2)$. The assertion
follows.
\end{proof}

\begin{thm}
\label{convexDiskThmC2}
Let $X$ be a finite digital image in~$\Z^2$ such that
$Bd(X)=\bigcup_{i=1}^n S_i$ is the disjoint union of 
$c_2$-closed curves $S_i$. Let $B_1$ be the set of points $x \in Bd(X)$ such that
$x$ is an endpoint of a maximal
slanted edge in $Bd(X)$. Let $B_2$ 
be the union of maximal horizontal 
and maximal vertical line segments in $Bd(X)$.
Let $B = B_1 \cup B_2$. Then
$B$ is a freezing set for $(X,c_2)$.
\end{thm}

\begin{proof}
Let $f \in C(X,c_2)$ such that $f|_B = \id_B$.

Let $p$ be a point of a slanted edge $E$ of
$Bd(X)$ such that $p \not \in B_1$.
Let $s$ and $s'$ be the endpoints of $E$.
If $f(p) \neq p$, it follows from
Lemma~\ref{cuPulling} that either $f(s) \neq s$
or $f(s') \neq s'$, a contradiction since by
hypothesis we have $\{s,s'\} \subset \Fix(f)$.
Therefore, $p \in \Fix(f)$; hence, every
slanted edge of $Bd(X)$ is a subset of $\Fix(f)$.
Since by hypothesis all horizontal and vertical edges
of $Bd(X)$ belong to $\Fix(f)$, we conclude that
$Bd(X) \subset \Fix(f)$. It follows from
Theorem~\ref{bdFreezes} that $f = \id_X$. Thus,
$B$ is a freezing set for $(X,c_2)$.
\end{proof}

\begin{thm}
\label{convDiskThmC2Actual}
Let $X$ be a thick convex disk with a  bounding
curve $S$. Let $B_1$ be the set of
points $x \in S$ such that
$x$ is an endpoint of a maximal
slanted edge in $S$. Let $B_2$ 
be the union of maximal horizontal 
and maximal vertical line segments in $S$.
Let $B = B_1 \cup B_2$. Then $B$ is a 
minimal freezing set for $(X,c_2)$
(see Figure~\ref{fig:convexRlts}(iii)).
\end{thm}

\begin{proof}
That $B$ is a freezing set follows as in
the proof of Theorem~\ref{convexDiskThmC2}. To show
$B$ is a minimal freezing set, we must show that
$B \setminus \{p\}$ is not a freezing set 
for every $p \in B$.

We start with $p \in B_1$.
Since $X$ is a convex disk, we only have the following 
possibilities to consider.
\begin{itemize}
       \item $X$ has an interior angle $\theta$ at $p$ of
          45$^\circ$ ($\pi / 4$ radians). Let
          $a \in X$ be such that $a \adj_{c_2} p$ and
          $a$ is adjacent to $p$ on an edge of $\theta$
          (see Figure~\ref{fig:degrees45}). Then 
          the function $f: X \to X$ given by
          \[ f(x) = \left \{ \begin{array}{ll}
              x & \mbox{if } x \neq p;  \\
              a & \mbox{if } x = p,
          \end{array}
          \right .
          \]
          belongs to $C(X,c_2)$, with 
          $X \setminus \{p\} = \Fix(f)$.
          Thus $B \setminus \{p\}$ is not
          a freezing set for $(X,c_2)$.
      \item $X$ has an interior angle at $p$ of
          90$^\circ$ ($\pi / 2$ radians). Then, there is a
          point $q \in Int(X)$ such that $p \adj_{c_1} q$
          as in Figure~\ref{fig:degrees90b},
          and the function $f: X \to X$ given by
          \[ f(x) = \left \{ \begin{array}{ll}
              x & \mbox{if } x \neq p;  \\
              q & \mbox{if } x = p,
          \end{array}
          \right .
          \]
          belongs to $C(X,c_2)$, with 
          $X \setminus \{p\} = \Fix(f)$.
          Thus $B \setminus \{p\}$ is not
          a freezing set for $(X,c_2)$.
         \item $X$ has an interior angle at $p$ of
          135$^\circ$ ($3 \pi / 4$ radians). Since $X$ is thick,
          there are points $a,b,b', q,q'$ as in 
          Figure~\ref{fig:degrees135c1}, i.e., $a$ and $q$
          are $c_2$-adjacent to $p$ along sides of the
          interior angle, such that 
          \[ N(X,p,c_2) = \{a,b,q,q'\} \subset
             N^*(X,b,c_2),
          \]
          and $\{a,b,p,q\} \subset N(X,c_2,b')$.
          Therefore, the function
          $f: X \to X$ given by
          \[ f(x) = \left \{ \begin{array}{ll}
              x & \mbox{if } x \neq p; \\
              b' & \mbox{if } x = p,
          \end{array}
          \right .
          \]
          belongs to $C(X,c_2)$, with
          $X \setminus \{p\} = \Fix(f)$. 
          Thus $B \setminus \{p\}$ is not
          a freezing set for $(X,c_2)$.
\end{itemize}

Now consider $p$ as a member of $B_2$. Since $X$ is convex,
this leaves only the following possibilities.
\begin{itemize}
       \item $X$ has an interior angle at $p$ of
          45$^\circ$ ($\pi / 4$ radians). Then
          $p \in B_1 \cap B_2 \subset B_1$. As discussed above,
          $B \setminus \{p\}$ is not a freezing set for $(X,c_2)$.
      \item $X$ has an interior angle at $p$ of
          90$^\circ$ ($\pi / 2$ radians). Let
          $a$ and $b$ be the points of the horizontal
          and vertical segments
          of $Bd(X)$ such that 
          $a \adj_{c_1} p  \adj_{c_1} b$ and let
          $q \in Int(X)$ be the point such that
          $a \adj_{c_1} q  \adj_{c_1} b$
          (see Figure~\ref{fig:degrees90a}). Then
          the function $f: X \to X$ defined by
          \[ f(x) = \left \{ \begin{array}{ll}
              x & \mbox{if } x \neq p; \\
              q & \mbox{if } x = p,
          \end{array}
          \right .
          \]
          is in $C(X,c_2)$ and 
          $\Fix(f) = X \setminus \{p\}$. So
          a freezing set for $(X,c_2)$ must 
contain $p$.
    \item $X$ has an interior angle at $p$ of
          135$^\circ$ ($3 \pi / 2$ radians). Then
          $p \in B_1 \cap B_2 \subset B_1$.
          As shown above, $B \setminus \{p\}$ is not
          a freezing set for $(X,c_2)$.
    \item $p$ is not an endpoint of its segment
          of $Bd(X)$. Then $p$ has a $c_1$-neighbor
          $q \in X$ 
          (see Figure~\ref{fig:notEndpt}).
       \begin{figure}
        \centering
        \includegraphics[height=2in]{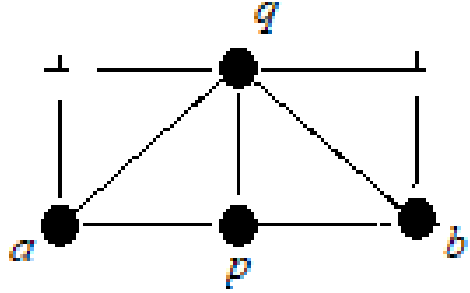}
        \caption{$p \in \overline{ab}$, a
        segment of the bounding curve $S$.
        $q \in Int(X)$.
           $p \adj_{c_1} q$.
            (Not meant to
        be understood as showing all of $X$.)
        }
        \label{fig:notEndpt}
        \end{figure}
          Then 
          the function $f: X \to X$ defined by
          \[ f(x) = \left \{ \begin{array}{ll}
              x & \mbox{if } x \neq p; \\
              b & \mbox{if } x = p,
          \end{array}
          \right .
          \]
          is in $C(X,c_2)$ and 
          $\Fix(f) = X \setminus \{p\}$. Hence
          $B \setminus \{p\}$ is not a freezing set for
          $(X,c_2)$.
 \end{itemize}

We have shown that for all $p \in B$,
$B \setminus \{p\}$ is not a freezing set for
$(X,c_2)$. Therefore, $B$ is a minimal
freezing set for $(X,c_2)$.
\end{proof}

\section{Further remarks}
Let $X$ be a thick convex digital disk 
in $\Z^2$ .
We have shown how to find minimal freezing sets
for $(X,c_1)$ and for $(X,c_2)$. 
We have given examples showing that our assertions
do not extend to non-convex disks in~$\Z^2$. However,
for non-convex disks in~$\Z^2$ we have shown how
to obtain smaller freezing sets than were
previously known.

We have left unanswered the following.

\begin{question}
Is every convex disk in $\Z^2$ thick?
\end{question}

\end{document}